\documentclass{amsart}
\usepackage{amsrefs}

\usepackage{amsfonts}
\usepackage{amsmath}
\usepackage[arrow,matrix,tips, curve]{xy}
\usepackage{amssymb}
\usepackage{amsthm}
\usepackage{latexsym}
\usepackage{url}
\usepackage{verbatim}

\usepackage{fancyref}

\theoremstyle{plain}
\newtheorem{theorem}{Theorem}
\newtheorem{proposition}[theorem]{Proposition}
\newtheorem{lemma}[theorem]{Lemma}
\newtheorem{corollary}[theorem]{Corollary}

\theoremstyle{definition}
\newtheorem{definition}{\mdseries\scshape Definition}

\theoremstyle{remark}

\newtheorem{remark}{\mdseries\scshape Remark}

\newtheorem{example}{\mdseries\scshape Example}

\newcommand{\set}[1]{\left\{#1\right\}}

\newcommand{\Z}{\mathbb{Z}}
\newcommand{\Zp}{\mathbb{Z}_p}

\newcommand{\oneunit}[1]{\left\langle#1\right\rangle}

\DeclareMathOperator{\ord}{ord}

\begin{document}

\title[Welch equation analysis]{Deconstructing the Welch Equation Using $p$-adic Methods}

\author{Abigail Mann}

\address{Department of Mathematics,
Rose-Hulman Institute of Technology,
5500 Wabash Avenue,
Terre Haute, IN 47803, USA}

\email{mannal@rose-hulman.edu}

\author{Adelyn Yeoh}

\address{Department of Mathematics and Statistics, 
Mount Holyoke College,
50 College Street, South Hadley,
  MA 01075, USA}

\email{yeoh22a@mtholyoke.edu}

\thanks{The first author would like to thank the Rose-Hulman Weaver and Rose Summer Undergraduate Research Programs, the Rose-Hulman Mathematics Department, and Dr. 
William Heller for their financial support, and Mount Holyoke College for their hospitality during her REU. The second author would like to thank the Mount Holyoke College Hutchcroft Fund for funding her summer REU}

\date{July 29, 2014}


\begin{abstract}
The Welch map $x \rightarrow g^{x-1+c}$ is similar to the discrete exponential map $x \rightarrow g^x$, which is used in many cryptographic applications including the ElGamal signature scheme. This paper analyzes the number of solutions to the Welch equation: $g^{x-1+c} \equiv x \pmod{p^e}$ where $p$ is a prime and $g$ is a unit modulo $p$, and looks at other patterns of the equation that could possibly be exploited in a similar cryptographic system. Since the equation is modulo $p^e$, where $p$ is a prime number, $p$-adic methods of analysis are used in counting the number of solutions modulo $p^e$. These methods include: $p$-adic interpolation, Hensel's lemma and Chinese Remainder Theorem. 
\end{abstract}

\maketitle
\section{Introduction}
The Welch equation: $g^{x-1+c} \equiv x \pmod{p^e}$ is typically used as an algorithm to produce Costas arrays, which are permutation matrices with certain desirable properties (see~\cite{rickard}). These arrays have applications in SONAR detection (see~\cite{drakakis}). The Welch equation is typically not associated with cryptography. However, the complexity associated with this equation may allow application in cryptography as well. In particular, we note that if viewed as a map, $x \rightarrow g^{x-1+c}$, the equation looks very similar to the discrete exponential map, $x \rightarrow g^x$. We follow a suggestion from Holden and Robinson (see~\cite[Section 8]{hold_rob}) to count the number of solutions to the Welch equation. 

In dissecting this problem, we first turn the equation into a function, $f_g(x,c):\Z \times \Z \rightarrow \Z/p^e\Z,$ where $f_g(x, c) = g^{x-1+c} - x \pmod{p^e}.$ When $g$ is fixed throughout a theorem we will suppress it and write $f(x,c)$. When $c$ is fixed, we write $f(\cdot ,c)$, and when $x$ is fixed, we write $f(x, \cdot )$. We then observe the output when $p$ is an odd prime. We also consider the special case when $p$ is equal to $2$. In the process leading up to counting the number of solutions, we observe that the output of $f(x, c)$ is periodic in $c$ and in $x$, allowing us to restrict the domain for these variables.

There are other characteristics of the Welch equation shown in Section~\ref{characteristics}. Most notably, when $g$ is a primitive root of modulo $p^e$, we find it easy to count the pairs of solutions $(x, c)$. The difficulty in this equation is how to count in the cases when $g$ is not a primitive root. Section~\ref{hintnonprg} describes a method to overcome this obstacle.

Our approach to counting solutions is to first begin with the function modulo $p$. In order to successfully count solutions modulo $p^e$, we will need to interpolate the equation. Section~\ref{interpolation} goes over the process of $p$-adic interpolation for our function. Finally in Section~\ref{numofsoln}, we count the solutions using different approaches to the equation: by fixing $c$ and treating $x$ as a variable, by treating both $x$ and $c$ as variables, and by looking at $p$ equal to $2$. 

For this paper, we assume that $g$ is a unit modulo $p$. Further, note that when $g$ is a unit, $x$ must also be a unit modulo $p$ in order to be a solution. Note that, unless otherwise stated, we take $m$ to be the multiplicative order of $g$ modulo $p$.

\section{Periodicity}

\begin{theorem}\label{thm:repetition}Let $f$ be the function $f(x, \cdot) = g^{x-1+c} - x \pmod{p^e}$, where $c \in \mathbb{Z}$, and $p$ is a prime. Then, we see that $f(x, c) = f(x, c + m \cdot p^{e-1})$.
\end{theorem}

\begin{proof} We want to prove the equivalent statement that $g^{m\cdot p^{e-1}} \equiv 1 \pmod{p^e}$. This is because we are able to reduce the statement by the following process:
\begin{eqnarray*}
f(x, c) &=& f(x, c+m\cdot p^{e-1}) \\
g^{x-1+c} - x &\equiv& g^{x-1+c + m \cdot p^{e-1}} - x \pmod{p^e}\\
g^{x-1+c} &\equiv& g^{x-1+c} \cdot g^{m \cdot p^{e-1}} \pmod{p^e}\\
g^{m \cdot p^{e-1}} &\equiv& 1 \pmod{p^e}
\end{eqnarray*}

Thus, it is sufficient to show that $g^{m \cdot p^{e-1}} \equiv 1 \pmod{p^e}$. By definition, we know that $g^m \equiv 1 \pmod{p}$ so we have $g^m = 1 + p\cdot A$, where $A \in \mathbb{Z}$. Now observe that by the binomial theorem we obtain:
\begin{eqnarray*}
g^{m \cdot p^{e-1}} &\equiv& (g^m)^{p^{e-1} } \pmod{p^e}\\
&\equiv& (1 + p \cdot A)^{p^{e-1}} \pmod{p^e}\\
&\equiv& 1 + p^{e-1}(p \cdot A) + \frac{p^{e-1}(p^{e-1} - 1)}{2!}(p \cdot A)^2 + \\ 
& & \frac{p^{e-1}(p^{e-1} - 1)(p^{e-1} - 2)}{3!}(p \cdot A)^3 +\ldots +\\
&& \frac{p^{e-1}(p^{e-1} - 1)\dots (p^{e-1} - n+1)}{n!}(p \cdot A)^n +\ldots +\\
&&  (pA)^{p^{e-1}} \pmod{p^e}.
\end{eqnarray*}
Consider the $n$th term of this expression:
\begin{eqnarray*}\frac{p^{e-1}(p^{e-1} - 1)\dots (p^{e-1} - n+1)}{n!}(p \cdot A)^n  \\ = p^e \left( \frac{p^{n-1}(p^{e-1} - 1)\dots (p^{e-1} - n+1)}{n!}(A)^n \right)\end{eqnarray*}
In this $n$th term,  for all $1 \leq n \leq p^{e-1}$, we have $p^e > n$. Thus, there are no multiples of $p^e$ in the denominator. Hence, the $n$th term is clearly divisible by $p^e$.

Now observe the last term, $(pA)^{p^{e-1}}$. By induction, we can easily show that $p^{e-1} \geq e$ for all $p \geq 2$, and for all $e \geq 1$. Thus with that result, $(pA)^{p^{e-1}}$ is a multiple of $p^e$, thus  $(pA)^{p^{e-1}} \equiv 0 \pmod{p^e}$.

Since the $n$th term and the last term are all divisible by $p^e$, it is then clear that $g^{m \cdot p^{e-1}} \equiv 1 \pmod{p^e}$.
\end{proof}

\begin{theorem} \label{thm:xrep}
Given a fixed $g$ and $c$, we have that
$$g^{x-1+c+mp^e}-(x+mp^e)  \equiv g^{x-1+c}-x \pmod{p^e}$$
In other words, $f(x,c) = f(x+mp^e,c)$.
\end{theorem}

\begin{proof}

From Theorem~\ref{thm:repetition} we now know that $g^{mp^{e-1}} \equiv 1 \pmod{p^e}$, so we have
\begin{eqnarray*}
g^{x-1+c+mp^e}-(x+mp^e) &\equiv& g^{x-1+c}g^{mp^{e-1}p}-x-mp^e \pmod{p^e}\\
&\equiv& g^{x-1+c}(1)^p-x \pmod{p^e}\\
&\equiv& g^{x-1+c}-x\pmod{p^e} 
\end{eqnarray*}
\end{proof}

\section{Characteristics of Welch Equation}\label{characteristics}
\subsection{When $g$ is a primitive root}

\begin{proposition}

Consider $f_g(x,c) = g^{x-1+c}-x \pmod{p^e}$ for odd prime p.
If $g$ is a primitive root mod $p^e$ (i.e. the order of $g$ is $p^{e-1}(p-1)$), there is a $c'$ for each c where
$$f_{g}(x,c)= -f_{g^{-1}}(p^{e+1}-x,c')$$ for all $x \in \Z$.
This corresponding $c'$ is given by $c' \equiv \frac{p^{e-1}(p-3)+4}{2}-c \pmod{p^{e-1}(p-1)}$ and is unique modulo $p^{e-1}(p-1)$.

\end{proposition}

\begin{proof}
To prove this $c'$ satisfies the requirements, it suffices to show $$f_{g}(x,c)+f_{g^{-1}}(p^{e+1}-x,c') \equiv 0 \pmod{p^e}.$$
So we have
\begin{eqnarray*}
g^{x-1+c}&-x+&(g^{-1})^{(p^{e+1}-x)-1+\frac{p^{e-1}(p-3)+4}{2}-c}-(p^{e+1}-x) \pmod{p^e} \\
&\equiv& g^{x-1+c}-x+g^{x-p^{e+1}+1+\frac{3p^{e-1}-p^e-4}{2}+c}+x \pmod{p^e} \\
 &\equiv& g^{x-1+c}(1+g^{-p^{e+1}+\frac{3p^{e-1}-p^e}{2}}) \pmod{p^e} \\
 &\equiv& g^{x-1+c}(1+g^{\frac{p^{e-1}(-2p^2-p+3)}{2}}) \pmod{p^e} \\
&\equiv& g^{x-1+c}(1+g^{\frac{p^{e-1}(p-1)(-2p-3)}{2}}) \pmod{p^e}. \\
\end{eqnarray*}
Since g is a primitive root modulo $p^e$, we know $g^{\frac{p^{e-1}(p-1)}{2}} \equiv -1 \pmod{p^e}$. So the expression we have reduces to:
$$\equiv g^{x-1+c}(1+(-1)^{-2p-3}) \pmod{p^e}.$$
Since $-2p-3$ will always be odd, we have that
$$ \equiv  g^{x-1+c}(1-1) \equiv 0 \pmod{p^e}.$$
Now suppose there is a $c'' \neq c'$ such that 
$$f_{g}(x,c) = -f_{g^{-1}}(p^{e+1}-x,c') = -f_{g^{-1}}(p^{e+1}-x,c'')$$
Then 
\begin{eqnarray*}
(g^{-1})^{(p^{e+1}-x)-1+c'}-(p^{e+1}-x) &\equiv& (g^{-1})^{(p^{e+1}-x)-1+c''}-(p^{e+1}-x) \pmod{p^e}
\end{eqnarray*}
Since $g$ and hence $g^{-1}$ are both primitive roots,
\begin{eqnarray*}
\\(p^{e+1}-x)-1+c' &\equiv&  (p^{e+1}-x)-1+c'' \pmod{p^{e-1}(p-1)}
\\c' &\equiv& c'' \pmod{p^{e-1}(p-1)}.
\end{eqnarray*}
Therefore, $f_{g}(x,c)=-f_{g^{-1}}(p^{e+1}-x,c')$ for a unique $c'$ modulo $p^{e-1}(p-1)$, where 
\\$c' \equiv \frac{p^{e-1}(p-3)+4}{2}-c \pmod{p^{e-1}(p-1)}$.

\end{proof}


\begin{lemma}\label{lem:uniquec} Let $p$ be an odd prime, and $g$ is a primitive root of $p^e$.  Then for each unit $x$ from $1$ to $p^e$, there exists a unique value of $c \in \{1, 2, \dots, (p-1)\cdot p^{e-1}\}$ that is a solution to $g^{x-1+c} \equiv x\pmod{p^e}$.
\end{lemma}

\begin{proof} Let $x \equiv g^k \pmod{p^e}$, so we obtain $g^{x-1+c} \equiv g^k \pmod{p^e}$. Now solve for $c$, and we obtain $c \equiv (k+1) - x \pmod{(p-1)\cdot p^{e-1}}$. Hence, we have shown that some unique $c$ exists, since it has to be in the range $1 \leq c \leq m\cdot p^{e-1}$.
\end{proof}

\begin{lemma}\label{lem:c_x}Let $p$ be an odd prime. Let $g$ be a primitive root of both $p$ and $p^e$. Then, when $x = p^e - 1$, we observe that
\begin{center} $c \equiv \frac{p^{e-1}(p-3)+4}{2} \pmod{p^{e-1}(p-1)}$
\end{center}
 is the solution to $g^{x-1+c} \equiv x \pmod{p^e}$.
\end{lemma}

\begin{proof}
We want to show that $g^{(p^e - 1)-1+\frac{p^{e-1}(p-3)+4}{2}} \equiv p^e - 1 \pmod{p^e}$.  
\begin{eqnarray*}
g^{(p^e - 1)-1+\frac{p^{e-1}(p-3)+4}{2}} &\equiv&  g^{ p^e - 2 + \frac{1}{2} p^{e-1}(p-3) + 2} \pmod{p^e} \\
&\equiv& g^{\frac{2p^e + p^{e-1}(p-3)}{2}} \pmod{p^e}\\
&\equiv& g^{p^{e-1}(p-1)} \cdot g^{\frac{p^{e-1}(p-1)}{2} }\pmod{p^e}\\
&\equiv& (1) \cdot (-1) \pmod{p^e}\\
&\equiv& p^e - 1 \pmod{p^e} \text{ }.
\end{eqnarray*}
\end{proof}

\begin{theorem}
Let $g$ be a primitive root modulo $p^e$, and let $g^{-1}$ be its multiplicative inverse. By Lemma~\ref{lem:uniquec}, there exists a pair ($x$, $c$) that solves $g^{x-1+c} \equiv x \pmod{p^e}$.

Then the pair ($p^e-x$, $c'$) solves $(g^{-1})^{p^e-x-1+c'} \equiv p^e-x \pmod{p^e}$, where \\$c' \equiv c_{p^e-1}-c \pmod{p^{e-1}(p-1)}$, and $c_{p^e-1}=\frac{p^{e-1}(p-3)+4}{2}$.

\end{theorem}

\begin{proof} By Lemma ~\ref{lem:uniquec} we know that there exists some $c'$ where $(p^e-x, c')$ is a solution to $(g^{-1})^{p^e-x-1+c'} \equiv p^e - x   \pmod{p^e}$.
Thus, we want to show that $c' \equiv c_{p^e-1} - c$. We know that $g^{x-1+c} \equiv x\pmod{p^e}$ and $g \cdot g^{-1} \equiv 1 \pmod{p^e}$. So we observe $(g^{-1})^{p^e-x-1+c'} \equiv p^e - x \pmod{p^e}$ and obtain:
\begin{align}
	(g^{-1})^{p^e-x-1+c'} &\equiv p^e - x \pmod{p^e}   \nonumber \\
	(g^{-1})^{p^e-x-1+c'} + g^{x-1+c} &\equiv p^e - x + x\pmod{p^e} \nonumber \\
	(g^{-1})^{p^e-x-1+c'} &\equiv -g^{x-1+c} \pmod{p^e} \nonumber \\
	(g\cdot g^{-1})^{p^e-x-1+c'} &\equiv -g^{(x-1+c) + (p^e-x-1+c')} \pmod{p^e} \nonumber \\
	1 &\equiv -g^{-2 + c + p^e + c'} \pmod{p^e} \nonumber \\
	p^e - 1 &\equiv g^{-2 + c + p^e + c'} \pmod{p^e} \label{eqnpe} \text{ } .
\end{align}

Lemma~\ref{lem:c_x} implies that $p^e - 1 \equiv g^{\frac{3p^{e-1}(p-1)}{2}} \pmod{p^e}$ by expanding the lemma as follows:
\begin{align}
	p^e - 1 &\equiv  g^{(p^e - 1)-1+\frac{p^{e-1}(p-3)+4}{2}}  \pmod{p^e} \nonumber \\
		&\equiv g^{\frac{2p^e +p^{e-1}(p-3)}{2}} \pmod{p^e} \nonumber \\
		&\equiv g^{\frac{3p^{e-1}(p-1) }{2} } \pmod{p^e} \label{eqnlem5} \text{ } .
\end{align}

Thus, we can equate (\ref{eqnpe}) and (\ref{eqnlem5}). So we obtain:
\begin{eqnarray*}
g^{\frac{3p^{e-1}(p-1) }{2} }  &\equiv& g^{-2 + c + p^e + c'} \pmod{p^e}\\
\frac{3p^{e-1}(p-1) }{2} &\equiv& -2 + c + p^e + c' \pmod{(p-1) \cdot p^{e-1}}\\
c' &\equiv& \frac{p^{e-1}(p-3)+4}{2} - c \pmod{(p-1) \cdot p^{e-1}} \text{ } .
\end{eqnarray*} 
\end{proof}

\subsection{For more general values of $g$ }
\begin{theorem}\label{thm:x0c0soln}
Let $p$ be an odd prime. Consider $x_0$ a unit. Let $(x_0, c_0)$ be a solution to $g^{x-1+c} \equiv x \pmod{p^e}$. Let $x_0 \equiv x_1 \equiv x_2 \equiv \ldots \equiv x_n \pmod{p^e}$ where 
\begin{eqnarray*}
x_1 &=& x_0 + 1\cdot p^e\\
x_2 &=& x_0 + 2 \cdot p^e\\
&\vdots&\\
x_n &=& x_0 + n \cdot p^e \text{ } .
\end{eqnarray*}
Then, 
\begin{eqnarray*}
c_n &\equiv& c_0 - n\cdot p^{e-1} \pmod{ m\cdot p^{e-1}}
\end{eqnarray*}
is the solution to the equation $g^{x_n-1+c_n} \equiv x_0 \pmod{p^e}$.
\end{theorem}

\begin{proof}
\begin{eqnarray*}
g^{x_n -1 + c_n} &\equiv& g^{ (x_0 + n \cdot p^e) - 1 + (c_0 - n \cdot p^{e-1})} \pmod{p^e}\\
&\equiv& g^{x_0 - 1 + c_0} \cdot g^{n\cdot p^{e-1} (p-1)} \pmod{p^e}\\
&\equiv& x_0 \pmod{p^e}
\end{eqnarray*}
\end{proof}

From Theorem~\ref{thm:x0c0soln} we explicitly highlight the case when we take $n = m$, and obtain the following corollary. We note that the statement is similar to Theorem~\ref{thm:repetition} except that this corollary highlights the particular $(x_0, c_0)$ pair that produces a repeated solution.

\begin{corollary}\label{corol:rptc}
Let $p$ be an odd prime. Let $(x_0, c_0)$ be a solution pair to $g^{x-1+c} \equiv x \pmod{p^e}$. Let $x_{m} = x_0 + m\cdot{p^e}$. Then, 
\begin{center}$c_m \equiv c_0 - m \cdot p^e \equiv c_0 \pmod{m\cdot p^{e-1}}$\end{center}
is the solution to the equation $g^{x_m - 1 + c_m} \equiv x_0 \pmod{p^e}$.
\end{corollary}

\subsection{Observing $x \equiv p \pmod{p}$}\label{hintnonprg}
\begin{proposition}\label{x=p}For all $c$ and $g \in \Z$, when $x \equiv p \pmod{p}$, $g^{x-1+c} \not\equiv x \pmod{p}$.
\end{proposition}

\begin{proof}It is easy to show why Proposition ~\ref{x=p} holds. If $x = p$, then $x \equiv 0 \pmod{p}$. However, $g^k$ can never be equal or congruent to 0. 
\end{proof}
When $x \equiv p \pmod{p}$ we get an interesting result. When we create the function $f(x, c) = g^{x-1+c} - x \pmod{p}$, and let $c$ range from $1$ to $m$ \footnote{We use this range because for a fixed $x$, the value of $f(x, c)$ starts to repeat when $c > m$. See Theorem~\ref{thm:repetition}. }, we obtain a value set, $V = \{f(x, c) \mid 1\leq c \leq m\}$. When $x \equiv p \pmod{p}$, the value set clues us in to the exact values of $x$ which have solutions to the Welch equation. This result is particularly useful to look in the case where $g$ is not a primitive root as not all values of $x$ will be a solution. Additionally, the size of the value set gives us insight into the number of solutions to modulo $p$ for the range of $x \in \{1, 2, \ldots, p\}$ . Hence, we have the following theorems.

\begin{lemma}\label{lem:vset}
Let $p$ be an odd prime. Let $f$ be a function defined by $f(x, c) = g^{x-1+c} - x \pmod{p}$. When $x = p$, then $f(p, c) \equiv g^c \pmod{p}$ for all $1 \leq c \leq m$.
\end{lemma}

\begin{proof} All we need to show is that $g^c \equiv g^{p-1+c} - p \pmod{p}$. This is fairly simple. Start with $LHS: g^{p-1+c} - p \equiv g^{p-1} \cdot g^c \equiv (1) \cdot g^c \equiv g^c \pmod{p}$.
\end{proof}

\begin{theorem}\label{thm:vset}Let $p$ be an odd prime, and fix $g$ a unit modulo $p$. Consider any $x \in \{ g^{p-1+c} - p \mid 1 \leq c \leq m\}$.
Then a solution $c'$ exists, which solves $g^{x-1+c'} \equiv x \pmod{p}$. Furthermore, this solution is unique modulo $m$.
\end{theorem}

The statement of this theorem may be a little confusing so we give an example before providing the proof. It is best to consider a $g$ that is not a primitive root, as this case best highlights the point of the theorem.

\begin{example}
Consider $p = 7$, and $g = 2$. The multiplicative order of 2 modulo 7 is 3. When $x = 7$ we obtain:
\begin{align}
c = 1 :&\text{  } 2^{7-1+1} - 7 \equiv 2^7 - 7 \equiv 128 - 7 \equiv 2 \pmod{7} \nonumber \\
c = 2 :&\text{  } 2^{7-1+2} - 7 \equiv 2^8 - 7 \equiv 256 - 7 \equiv 4 \pmod{7} \nonumber \\
c = 3 :&\text{  } 2^{7-1+3} - 7 \equiv 2^9 - 7 \equiv 512 - 7 \equiv 1 \pmod{7} \nonumber
\end{align}
So the value set, $V = \{1, 2, 4\}$. Check $f(x, c) = g^{x-1+c} - x \pmod{p}$ for $x \in \{1, 2, 3, 4, 5, 6, 7\}$, and for $c \in \{1, 2, 3\}$. When $f(x, c) \equiv 0 \pmod{p}$, we know that this pair $(x, c)$ is a solution to the equation $g^{x-1+c} \equiv x \pmod{p}$.


\begin{table}[h]
\begin{center}
\begin{tabular}{| l | c c c|}
\hline
\textbf{x} & \textbf{c = 1} & \textbf{c = 2} & \textbf{c = 3}\\
\hline
\textbf{1} & 1 & 3 & 0\\
\textbf{2} & 2 & 6 & 0\\
\textbf{3} & 5 & 6 & 1\\
\textbf{4} & 5 & 0 & 4\\
\textbf{5} & 6 & 3 & 4\\
\textbf{6} & 2 & 3 & 5\\
\textbf{7} & 1 & 2 & 4\\
\hline
\end{tabular}
\caption{Values of $f(x, c)$}\label{tablevset}
\end{center}
\end{table}
As we can see in Table~\ref{tablevset}, the values of $x$ which have solutions are the same as those in the value set. 
\end{example}

\begin{proof}
By Lemma~\ref{lem:uniquec} we know that if $x$ is a solution, there will exist a corresponding $c$. Let $c'$ be the solution to $g^{x-1+c} \equiv x \pmod{p}$ for $x \equiv g^{p-1+c} - p \pmod{p}$. So we have the following equivalences:
	\begin{align*}
	g^{x-1+c'} &\equiv x \pmod{p} \\
	g^{x-1+c'} &\equiv  g^{p-1+c} - p \pmod{p}\\
	x-1 + c' &\equiv c \pmod{m} \\
	c' &\equiv  -x + 1 - c \pmod{m}
	\end{align*}

Now we show that $c'$ solves $g^{p-1 + c } \equiv x \pmod{p}$. 
 
\begin{align*}
g^{x-1+c'} = g^{x-1 + (-x + 1- c)} \equiv g^c \equiv g^c\cdot g^{p-1} - p \equiv g^{p-1+c} - p \equiv x \pmod{p}
\end{align*}

\end{proof}

\subsection{Symmetry in multiplicative inverses}
\begin{proposition}\label{prop:backwdxp}
Let $p$ be an odd prime. If $g^{-1}$ is the multiplicative inverse of $g$ modulo $p$, then when $x = p$, we have $g^{p-1+c} \equiv (g^{-1})^{p-1+ (m-c)} \pmod{p}$.
\end{proposition}

\begin{proof}	
\begin{eqnarray*}
(g^{-1})^{p-1 + (m-c)} 
&\equiv& g^{-(p-1 + (m-c))} \pmod{p} \\
&\equiv& g^{-(p-1)} \cdot g^{-m} \cdot g^c \pmod{p} \\
&\equiv& (1) \cdot (1) \cdot g^c \pmod{p}\\
&\equiv& g^{p-1+c} \pmod{p},
\end{eqnarray*}
where the last equivalence follows from Lemma ~\ref{lem:vset}.
\end{proof}

\begin{theorem}
Let $p$ be an odd prime. Let $f$ be the function $f_g(p,c) = g^{p-1+c} - p \pmod{p}$. When $g$ and $g^{-1}$ are multiplicative inverses we observe that the value sets produced are equal, such that $\{f_g(p,c) \mid 1 \leq c \leq m \} = \{ f_{g^{-1}}(p,c) \mid 1 \leq c \leq m\}$.
\end{theorem}

\begin{proof}
By Proposition~\ref{prop:backwdxp}, we know that $g^{p-1+c} \equiv  (g^{-1})^{p-1+ (m-c)} \pmod{p}$. Hence, if we consider $f_g(p,c)$ and take it over $1 \leq c \leq m$, we know that the value set will be equal to the value set of $f_{g^{-1}}(p,c)$.
\end{proof}


\subsection{Other patterns}
\begin{proposition}
Consider $f(x,c) \equiv g^{x-1+c}-x \pmod{p^e}$. Then
$$f(x+y, c) \equiv f(x, c+y)-y \pmod{p^e}.$$
\end{proposition}

\begin{proof}
It suffices to show $f(x+y, c) - (f(x,c+y)-y) \equiv 0 \pmod{p^e}$.
We have 
\begin{eqnarray*}
g^{(x+y)-1+c}&-&(x+y)-(g^{x-1+c+y}-x-y) \\
&\equiv& g^{x+y-1+c}-g^{x+y-1+c}-x-y+x+y \pmod{p^e}\\
&\equiv& 0 \pmod{p^e}
\end{eqnarray*}
\end{proof}

\begin{proposition}
Consider $f(\cdot, c) \equiv g^{x-1+c}-x \pmod{p^e}$. Then
$$f(x, c) \equiv f(x+p^{e-1}(p-1),c)-p^{e-1} \pmod{p^e}.$$
\end{proposition}

\begin{proof}
It suffices to show $f(x,c) - f(x+p^{e-1}(p-1),c)+p^{e-1} \equiv 0 \pmod{p^e}.$
The left hand side becomes
\begin{equation*}
g^{x-1+c} - x -(g^{x+p^{e-1}(p-1)-1+c}-x-p^{e-1}(p-1)-p^{e-1}) 
\end{equation*}
Since the order of $g$ divides $p^{e-1}(p-1)$, this is congruent modulo $p^e$ to
\begin{eqnarray*}
 g^{x-1+c}&-&x-g^{x-1+c}+x+p^{e-1}(p) \pmod{p^e}\\
&\equiv& p^e \pmod{p^e}\\
&\equiv& 0 \pmod{p^e}.
\end{eqnarray*}
\end{proof}

\section{Interpolation $\&$ Hensel's Lemma}\label{interpolation}
\subsection{When $p$ is an odd prime}

Let $g \in \Z$ be fixed and let $p$ be an odd prime.  
We will need to interpolate the function $f(\cdot, c) = g^{x-1+c}$,
which is defined on $x \in \Z$, to a function on $x \in \Zp$
so that we can count solutions to $g^{x-1+c} \equiv x \pmod{p^e}$.
However, this is not possible for $g \notin 1+p\Zp$. 
(See for example, \cite[Section 4.6]{gouvea}, or \cite[Section II.2]{koblitz}.)
So we will have to change the function slightly in order to interpolate.

To do this, we use methods similar to that of \cite[Section 2]{hold_rob}.
Let $\mu_{p-1} \subseteq \Zp^\times$
be the set of all $(p-1)$-st roots of unity.  Then for odd prime $p$, we have the
Teichm\"uller character
$$\omega: \Zp^\times \to \mu_{p-1},$$
which is a surjective homomorphism.
As stated in \cite[Cor. 4.5.10]{gouvea},
we can write each element of $\Zp^\times$
as an element of  $\mu_{p-1} \times (1+p\Zp)$. 
So for each $x \in \Z_p^\times$ as $\Z_p^\times \cong \mu_{p-1} \times (1+p\Zp)$  we have $x = \omega(x) \oneunit{x}$ for some
$\oneunit{x} \in 1 + p\Zp$.

\begin{theorem}
\label{gouvea-interp}
For $p \neq 2$, let $g \in \Zp^\times$ and $x_0 \in \Z/(p-1)\Z$, and let
$$I_{x_0} = \set{x \in \Z \mid x-1+c \equiv x_0
  \pmod{p-1}} \subseteq \Z.$$
Then $$F_{x_0}(x) = \omega(g)^{x_0}\oneunit{g}^{x-1+c}$$ defines a uniformly continuous function
on $\Zp$ such that $F_{x_0}(x) = g^{x-1+c}$ whenever $x\in I_{x_0}$.
\end{theorem}

\begin{proof}
By ~\cite[Proposition 4.6.1]{gouvea}, we know that we need
 $I_{x_0}$ to be dense in $\Z_p$ and for each $F_{x_0}(x)$ be uniformly continuous and bounded.
By ~\cite[Theorem 4.1.4]{katok}, we know that if a function is continuous in $\Z_p$, it is also uniformly continuous and bounded.
Thus, it suffices to show density of $I_{x_0}$, continuity of each $F_{x_0}$ as a function on $I_{x_0}$, and that $F_{x_0}(x)=g^{x-1+c}$ with the proper conditions on $x$.

We first prove density of $I_{x_0}$ in $\Z_p$: for the sake of clarity, we rename $I_{x_0}$ as $I_{s_0} = \set{s \in \Z \mid s-1+c \equiv s_0 \pmod{p-1}}$ in this proof.
Let $x \in \Z_p$, so it can be written as $x = x_0 + x_1p+x_2p^2 + \ldots$. Now let
 $s_1 \in \{0, 1, \ldots, p-2\}$ such that $s_1 \equiv s_0+1-c-x_0-x_1- \ldots -x_{i-1} \pmod{p-1}$
and $y_i=x_0+x_1p+\ldots+x_{i-1}p^{i-1}+s_1p^i$ for any given $i \in \Z$. Then we know 
$$y_i \equiv x_0+x_1+x_2+\cdots+x_{i-1}+s_1  \equiv s_0+1-c \pmod{p-1}$$ so $y_i \in I_{s_0}$ for all $y_i$.
Now let $\epsilon>0$, and we can find $N$ such that $p^{-N}<\epsilon$. For any $n>N$, we have
\begin{eqnarray*}
\lvert y_n-x \rvert_p &=& \lvert x_0+x_1p+\ldots+x_{n-1}p^{n-1}+s_1p^n-(x_0+x_1p+x_2p^2+\ldots)\rvert_p
\\ &=& \lvert s_1p^n-(x_np^n+x_{n+1}p^{n+1}+\ldots)\rvert_p
\\ &=& \lvert p^n \rvert_p \lvert s_1 - (x_n+n_{n+1}p+ \ldots)\rvert_p
\\ &\leq& p^{-n} <p^{-N} < \epsilon.
\end{eqnarray*}
So for every $x \in \Z_p$ we have a sequence $\{y_i\}$ in $I_{s_0}$ that converges to $x$, so then $I_{s_0}$
 (which was the new notation for our original set $I_{x_0}$) is dense in $\Z_p$.

Now we must show each $F_{x_0}(x) = \omega(g)^{x_0}\oneunit{g}^{x-1+c}$ is uniformly continuous on $I_{x_0}$.
Given $\epsilon>0$, find $N$ such that $p^{-N}<\epsilon$.
 Now if $x,y \in I_{x_0}$ such that 
$$\lvert x-y \rvert_p \leq p^{-N} < p^{-(N-1)} = \delta,$$ then $x=y+p^NA$ for some $A \in \Z$.
 Consider
$$\lvert\oneunit{g}^x-\oneunit{g}^y\rvert_p = \lvert\oneunit{g}^{y+p^NA}-\oneunit{g}^y\rvert_p =
 \lvert\oneunit{g}^y\rvert_p \lvert\oneunit{g}^{p^NA}-1\rvert_p = \lvert\oneunit{g}^{p^NA}-1\rvert_p$$
and using the binomial theorem for some $M \in \Z$, we get
$$\oneunit{g}^{p^NA} = (1+pM)^{p^NA} = 1+p^NApM+\frac{p^NA(p^NA-1)}{2}(pM)^2+\ldots+(pM)^{p^NA}.$$
Because all terms except for the first are in $p^{N+1}\Z_p$, we see that 
$$\lvert\oneunit{g}^{p^NA}-1\rvert_p \leq p^{-(N+1)} < p^{-N} < \epsilon.$$
So the function mapping $x \rightarrow \oneunit{g}^x$ is uniformly continuous on $I_{x_0}$ and hence on $\Z_p$ by ~\cite[Thm 4.15]{katok}.
Since each $F_{x_0}(x) = \omega(g)^{x_0}\oneunit{g}^{x-1+c} = \omega(g)^{x_0}\oneunit{g}^{c-1}\oneunit{g}^x$
for fixed $x_0, c$, and $g$,  and $ \omega(g)^{x_0}\oneunit{g}^{c-1}$ is a constant, we have that $F_{x_0}(x)$ is a constant 
times a uniformly continuous function. Hence, each $F_{x_0}(x)$ is uniformly continuous on $\Z_p$ (see ~\cite[Exercise 89]{katok}).

Lastly, we show that $F_{x_0}(x) = g^{x-1+c}$ when $x \in I_{x_0}$.
Since $x-1+c \equiv x_0 \pmod{p-1}$, we have that
 $$g^{x-1+c} = \omega(g)^{x-1+c}\oneunit{g}^{x-1+c} = \omega(g)^{x_0}\oneunit{g}^{x-1+c} = F_{x_0}(x).$$

\end{proof}

We can extend this theorem to multiples of the order of $g$ modulo $p$:

\begin{theorem}
  For this theorem, we let $m$ be any multiple of the multiplicative order of $g$ modulo
  $p$, $p \neq 2$, so that $m \mid p-1$.
  Let $g \in \Zp^\times$ and $x_0 \in \Z/m\Z$, and let
$$I_{x_0} = \set{x \in \Z \mid x-1+c \equiv x_0
  \pmod{m}} \subseteq \Z.$$
Then $$F_{x_0}(x) = \omega(g)^{x_0}\oneunit{g}^{x-1+c}$$ defines a uniformly continuous function
on $\Zp$ such that $F_{x_0}(x) = g^{x-1+c}$ whenever $x\in I_{x_0}$.
\end{theorem}

\begin{proof}
Since $g^m \equiv 1 \pmod{p}$, $\omega(g)^m = \omega(g^m) = \omega(1) = 1.$ 
If $x_0, x_0' \in \Z/(p-1)\Z$ and $ x_0 \equiv x_0' \pmod{m}$,
then the two functions $F_{x_0}$ and $F_{x_0'}$ given by
Theorem~\ref{gouvea-interp} are equal and are the same as $g^{x-1+c}$ when
$x \in I_{x_0} \cup  I_{x_0'}$.
\end{proof}

\subsection{When $p$ is equal to 2}
In the case that $p=2$, we use the same ideas to interpolate the function $f(\cdot, c) = g^{x-1+c}-x \pmod{p^e}$ a little differently. 
We will still decompose $g \in \Z_2^\times$ as $\omega(g)\oneunit{g}$, but now we take
$\omega(g) \in \{-1, 1\}$ and $\oneunit{g} \in 1+4\Z_2$. It is known that this factorization
exists and is unique (see ~\cite[Corollary 4.5.10]{gouvea}).

Now we will show how we can form two new functions by interpolation when $p=2$.

\begin{theorem}
\label{p2 interp}
For $p = 2$, let $g \in \Z_2^\times$ and $c \in \Z$.
Then $F_0(x) = \oneunit{g}^{x-1+c}$ and
$F_1(x) =  -\oneunit{g}^{x-1+c}$
 define functions
on  $1+2\Z_2$ such that $F_0(x) = g^{x-1+c}$ either when
  $g \in 1+4\Z_2$ or when $g \in 3+4\Z_2$ and $x-1+c \equiv 0 \pmod{2}$
and $F_1(x) = g^{x-1+c}$ when $g \in 3+4\Z_2$ and $x-1+c \equiv 1 \pmod{2}$.
\end{theorem}

\begin{proof}
As with the proof for odd primes (see Theorem~\ref{gouvea-interp}), it suffices to show that $1+2\Z$ is dense in $1+2\Z_2$,
each function is uniformly continuous on $1+2\Z$, and the functions agree with $g^{x-1+c}$ for the 
proper conditions on $x$.

The density of $1+2\Z$ in $1+2\Z_2$ is simple to show. For any $x \in 1+2\Z_2$,
$$x = 1 + a_1(2)+a_2(2^2)+a_3(2^3)+ \ldots \quad \mbox{where } a_i \in \{0, 1\}.$$ 
We let $\{y_n\}$ be the sequence defined by $y_n = x \pmod{2^n}$.
Since each $y_n \in 1+2\Z$ and $\{y_n\}$ converges to $x$, we know 
that $1+2\Z$ is dense in the set $1 + 2\Z_2$.

Next,  we show that each function is uniformly continuous on $1+2\Z$.
Let $\epsilon > 0$. We take $N$ such that $\frac{1}{2^N}<\epsilon$.
Now let $x,y \in 2\Z+1$, and take $\delta > 0$ such that
 $$\lvert x-y \rvert_2 \leq \frac{1}{2^N} < 2^{-(N-1)}=\delta$$
Then $x-y \in 2^N\Z$, so $x=y+2^NA$ for some $A \in \Z$. So
$$ \lvert \oneunit{g}^y-\oneunit{g}^x\rvert_2 = \lvert \oneunit{g}^y-\oneunit{g}^{y+2^NA}\rvert_2 
= \lvert \oneunit{g}^y\rvert_2  \lvert 1-\oneunit{g}^{2^NA}\rvert_2 = \lvert \oneunit{g}^{2^NA}-1\rvert_2 $$
Notice that for $g=1+4M$ with $M \in Z_2$ we have 
$$\oneunit{g}^{2^NA} = (1+4M)^{2^NA} = 1+2^NA(4M)+\frac{2^NA(2^NA-1)}{2}(4M)^2+\ldots+(4M)^{2^NA}$$
and all terms except the first are in $2^{N+1}\Z_2$.
So $\lvert\oneunit{g}^{2^NA}-1\rvert_2 \leq 2^{-(N+1)} < 2^{-N}< \epsilon.$
\\So $\oneunit{g}^x$ is uniformly continuous on $1+2\Z$.
Since $\oneunit{g}^{-1+c}$ and  $-\oneunit{g}^{-1+c}$ are constants, both $F_0(x)$ and $F_1(x)$ 
are uniformly continuous on $1+2\Z_2$ as well by interpolation.

Now we have that for fixed $c \in \Z$ and $g \in \Z_2^\times$,
 $g^{x-1+c}=\omega(g)^{x-1+c}\oneunit{g}^{x-1+c}$. 
 If $g \in 1+4\Z_2$ then 
$g^{x-1+c}=1^{x-1+c}\oneunit{g}^{x-1+c}=\oneunit{g}^{x-1+c}$.
If $g \in 3+4\Z_2$, then 
$g^{x-1+c}=(-1)^{x-1+c}\oneunit{g}^{x-1+c}$. Since $(-1)^2=1$, we have two cases.
Suppose $x-1+c \equiv 0 \pmod{2}$. Then $(-1)^{x-1+c}=1$ and $g^{x-1+c}=\oneunit{g}^{x-1+c}$.
If $x-1+c \equiv1 \pmod{2}$, then $(-1)^{x-1+c}=-1$ and $g^{x-1+c}=-\oneunit{g}^{x-1+c}$.
So we have two equations for $f(\cdot, c)=g^{x-1+c}$. For $g \in 1+4\Z_2$ or $g \in 3+4\Z_2$ and $x-1+c \equiv 0 \pmod{2}$
$$ F_0(x) = \oneunit{g}^{x-1+c}$$
Otherwise (when $g \in 3+4\Z_2$ and $x-1+c \equiv 1 \pmod{2}$),
$$ F_1(x) = -\oneunit{g}^{x-1+c}.$$

\end{proof}

We have thus shown that we can interpolate our function $f(\cdot, c) \equiv g^{x-1+c} - x \pmod{p^e}$. Since our goal is to count solutions, we find that we can form a power series to conduct further analysis of the function, where we use Hensel's lemma to discuss how our function is able to ``lift" from solutions  modulo $p$ to solutions modulo $p^e$. We use the generalizations of Hensel's lemma from~\cite[Section 3]{hold_rob} in our analysis. The reader will see use of Hensel's lemma throughout Section~\ref{numofsoln} in finding the number of solutions modulo $p^e$ from those modulo $p$.

\section{Counting solutions}\label{numofsoln}
\subsection{Treating $c$ as a fixed constant}

In the proof of the following theorem, we use the $p$-adic $\log$ and $\exp$ functions. They are defined as follows: 

$$\exp(x) = \sum\limits_{n=0}^{\infty}\frac{x^n}{n!}, \hspace{5 mm} \lvert x \rvert_p < p^{-1/(p-1)}
$$
$$ \log(1+x) = \sum\limits_{n=1}^{\infty}\frac{(-1)^{n+1}x^n}{n}, \hspace{5 mm} \lvert x \rvert_p < 1$$ 

It is important to note that for $x$ where everything is defined, $\exp(\log(1+x))=1+x$ and $\log(\exp(x))=x$. For more on these functions, see ~\cite[Section 4.5]{gouvea}.

\begin{theorem} \label{x0fixedpt} For $p \neq 2$, let $g \in \Zp^\times$ be fixed (and again let $m$ be the
  multiplicative order of $g$ modulo $p$).
  Then for each $x_0 \in \Z/m\Z$, there is exactly one
  solution to the equation
$$\omega(g)^{x_0} \oneunit{g}^{x-1+c} = x $$
for $x \in \Zp$.
\end{theorem}

\begin{proof} 
Similarly to the proof found in ~\cite[Section 4]{hold_rob},
we start by finding solutions modulo $p$.  Since
  $\oneunit{g} \equiv 1 \pmod{p}$, we now have
$$\omega(g)^{x_0} \equiv x \pmod{p}.$$
We fixed $g$ and $x_0$, so this clearly has exactly one solution.

Since
  $\oneunit{g}$ is in $1+p\Zp$, we get
\begin{eqnarray*}
\oneunit{g}^{x-1+c}=\oneunit{g}^{c-1}\oneunit{g}^x&=&
\oneunit{g}^{c-1}(\exp(x \log(\oneunit{g})) \\ = \ 
\oneunit{g}^{c-1}\  (1&+&x\log(\oneunit{g})+
x^2\log( \oneunit{g})^2/2! \\
&+& \mbox{higher order terms in powers of }
\log(\oneunit{g}))
\end{eqnarray*}
where from ~\cite[Proposition 4.5.9]{gouvea},  we know that 
$\log(\oneunit{g}) \in p\Zp$. Now that we have a convergent
power series since $|\log(\oneunit{g})^i/i!|_p \to 0$ as $i \to \infty$ (see ~\cite[Chapter 2, Thm 3.1]{bachman}), we 
examine $f(\cdot,c)=F_{x_0}(x)-x$ and its derivative to see if we can apply a generalization of Hensel's lemma.

We let $a \in \{0, 1, \ldots, p-1\}$ such that $a \equiv \omega(g)^{x_0} \pmod{p}$ and let
\begin{eqnarray*}
f(\cdot,c)=\omega(g)^{x_0}\oneunit{g}^{c-1}(1&+&x\log(\oneunit{g})+
x^2\log( \oneunit{g})^2/2! \\
&+& \mbox{higher order terms in powers of }
\log(\oneunit{g})) - x
\end{eqnarray*}

Since we know $\log{(\oneunit{g})}\in p\Zp$, so $\log{(\oneunit{g})} \equiv 0 \pmod{p}$, we have that
\begin{eqnarray*}
f(a,c)\equiv\omega(g)^{x_0}(1)(1&+&a(0)+
a^2(0) \\
&+& \mbox{higher order terms congruent to 0 (mod p) }
) -  a \pmod{p}
\end{eqnarray*}
\begin{eqnarray*}
\equiv\omega(g)^{x_0}-a
\equiv 0 \pmod{p}
\end{eqnarray*}
And we also have that
\begin{eqnarray*}
f'(\cdot,c)=(\omega(g)^{x_0}\oneunit{g}^{c-1})(\log{(\oneunit{g})}+x\log{(\oneunit{g})}^2+x^2\log{(\oneunit{g})^3/2!}+...)-1
\end{eqnarray*}
so that
\begin{eqnarray*}
f'(a,c)=(\omega(g)^{x_0}\oneunit{g}^{c-1})(\log{(\oneunit{g})}+a\log{(\oneunit{g})}^2+...)-1 \equiv -1 \not\equiv 0 \pmod{p}.
\end{eqnarray*}
which is also convergent (see \cite[Proposition 4.4.4]{gouvea}).
Now we know we can apply a generalization of Hensel's lemma (see ~\cite[Cor. 3.3]{hold_rob}), 
 which states that there is a unique $x \in \Z_p$ for which $x \equiv a \pmod{p}$ and $f(x,c)=0$ in $\Z_p$.

\end{proof}

\begin{corollary} \label{crtfixedpt} For odd prime $p$, let $g \in \Z$ be fixed where $p \nmid g$.
  Then there are exactly $m$ solutions to the congruence
\begin{equation*}
g^{x-1+c} \equiv x \pmod{p^e} 
\end{equation*}
for $x \in \{1,2, \ldots, p^em \}$.  These solutions  are also all distinct modulo $m$ and relatively prime to $p$.
\end{corollary}

\begin{proof}
Theorem~\ref{x0fixedpt} implies that for each $x_0 \in \Z/m\Z$ there is exactly one
$x_1 \in \Z/p^e\Z$ 
where
$$\omega(g)^{x_0} \oneunit{g}^{x_1-1+c} \equiv x_1 \pmod {p^e}. $$
The Chinese Remainder Theorem states that
there will be exactly one  $x \in \Z/p^em\Z$ where
${x-1+c} \equiv x_0 \pmod m$ and $x \equiv x_1 \pmod {p^e}$. Since $x-1+c \equiv x_0 \pmod m$, we know that
for this $x$:
$$g^{x-1+c} =\omega(g)^{x_0} \oneunit{g}^{x-1+c} \equiv x \pmod {p^e}. $$
Since for each $x_0$ there is exactly one such $x$, we have exactly $m$ solutions to the congruence.
Note that, as stated in the introduction, $x$ must be a unit modulo $p$ since $g$ is a unit, and thus all solutions $x$ are relatively prime to $p$.

\end{proof}

Furthermore, we can extend our knowledge of the number of solutions to a larger range of $x$.
\begin{proposition}
 For an odd prime $p$, let $g \in \Z$ be fixed such that $p \nmid g$, and let $k \in \Z$.
  Then there are exactly $km$ solutions to the congruence
\begin{equation}
g^{x-1+c} \equiv x \pmod{p^e} 
\end{equation}
for $x \in \{1,2, \ldots, p^ekm \}$, each $x$ relatively prime to $p$.
Note that the order of $g$ modulo $p$ must always divide $p-1$, so when $k= \frac{p-1}{m}$ there are $p-1$ solutions for $x \in \{1, 2, \ldots, (p-1)p^e\}$.
\end{proposition}

\begin{proof}
From Theorem~\ref{thm:xrep}, we notice that
$$g^{x-1+c+mp^e}-(x+mp^e) \equiv g^{x-1+c}-x \pmod{p^e}$$
So we know that $g^{x-1+c} \equiv x \pmod{p^e}$
 has the same number of solutions
 for $x \in \{1,2, \ldots, p^em \}$  as it does
for $x \in \{k_1m+1, k_1m+2, \ldots, k_1m+p^em\}$ for any $k_1 \in \Z$.
Since $\{1, 2, \ldots, kp^em\} = \{1, 2, \ldots, p^em\} \cup \{p^em+1, p^em+2,\ldots, 2p^em\} \cup \ldots \cup \{(k-1)p^em+1, (k-1)p^em+2, \ldots, kp^em\}$,
and we know that there are $m$ solutions for $x \in \{1,2, \ldots, p^em \}$ from Corollary~\ref{crtfixedpt},
then the number of solutions for $x \in \{1, 2, \ldots, p^ekm\}$ is equal to $km$, and each solution $x$ is relatively prime to $p$ since $g$ is a unit.

\end{proof}



\subsection{Treating both $x$ and $c$ as variables}

One difficulty with the Welch equation is to ensure that the domains of $x$ and $c$ modulo $p$ will scale nicely modulo $p^e$. A key problem is that we cannot predict how the multiplicative order of $g$ modulo $p$ changes when $g$ is considered modulo $p^e$. Generally, the multiplicative order of $g$ modulo $p^e$ is simply $p^{e-1} \cdot \ord_{p}(g)$. When this happens, we can predict the exact period (see Theorem~\ref{thm:repetition}). 

However, it is not always the case that the multiplicative order of $g$ modulo $p^e$ is $p^{e-1}\ord_{p}(g)$. Sometimes, the multiplicative order of $g$ modulo $p^e$ is equal to the multiplicative order of $g$ modulo $p^{e-1}$.
\begin{example}
Let $p = 11$ and $g= 3$. Let $e = 1$, and take powers of $g$ as follows:  $3^1 \equiv 3 \pmod{11}$, $3^2 \equiv 9 \pmod{11}$, $3^3 \equiv 5 \pmod{11}$, $3^4 \equiv 4 \pmod{11}$, $3^5 \equiv 1 \pmod{11}$. Thus, the multiplicative order of $3$ modulo $11$ is $5$.

Now take $e = 2$, and take powers of $g$ as before. Observe that $3^5 \equiv 1 \pmod{11^2}$. Thus, the multiplicative order of $3$ modulo $11^2$ is $5$.
\end{example}

When the multiplicative orders are equal, the exact period of $c$ modulo $p^e$ is shorter than we expect. To account for such a problem we have the following remark followed by a theorem.

\begin{remark}\label{lem:logexistornot}
Let $p$ be an odd prime, and fix $g$. Consider $a \in \mathbb{Z}$. Then $\log_g(a)$ may or may not exist.
\end{remark}

\begin{theorem}
Let $p$ be an odd prime. Given a fixed $x$, the number of solutions, $c$, to $g^{x-1+c} \equiv x \pmod{p^e}$ for $c \in \{1, 2, \ldots, m\cdot p^{e-1} \}$ is either $\frac{mp^{e-1}}{\ord_{p^e}(g)}$ when $\log_g(a)$ exists or $0$ when $\log_g(a)$ does not exist.
\end{theorem}
\begin{proof}
Let $x \equiv a \pmod{p^e}$, and by Remark~\ref{lem:logexistornot} we note that there are two cases: when $\log_g(a)$ exists, and when $\log_g(a)$ does not exist. Now if $\log_g(a)$ exists then we have:
\begin{align}
g^{a-1+c} &\equiv a \pmod{p^e} \nonumber \\
a-1+c &\equiv \log_g(a) \pmod{\ord_{p^e}(g)} \nonumber \\
c &\equiv (\log_g(a)) - a + 1\pmod{\ord_{p^e}(g)}~\label{eqn0mpe1}
\end{align}

Thus, when $\log_g(a)$ exists, we can solve equation~\ref{eqn0mpe1} modulo $\ord_{p^e}(g)$, since we are looking at values for $c \in \{1, 2, \ldots, mp^{e-1} \}$. There will be a solution for $c$ every multiple of $\ord_{p^e}(g)$. Hence, the number of possible values for $c$ is $\frac{mp^{e-1}}{\ord_{p^e}(g)}$. 

When $\log_g(a)$ does not exist, we are not able to solve for $c$ in equation~\ref{eqn0mpe1}. For this case, we are not able to find $c$. Thus, the number of possible values for $c$ is $0$.
\end{proof}

Now we try to count the number of $(x, c)$ pairs of solutions modulo $p^e$. Proposition~\ref{miniproofxc} gives the number of pairs of solutions modulo $p$. Theorem~\ref{thm:finalxcsoln} gives the number of pairs of solutions modulo $p^e$ by using a multivariable Hensel's lemma from~\cite[Proposition 3.4]{hold_rob} together with the Chinese Remainder Theorem.

\begin{definition}
 We let $T_e$ denote the set of solution pairs $(x, c)$, where $x \in \{1, 2, \ldots, mp^e \}$ and $c \in \{1, 2, \ldots, mp^{e-1} \}$, $p \nmid x$ to the equivalence:
\begin{equation*}
g^{x-1+c} \equiv x \pmod{p^e}.
\end{equation*}
Also, let $|T_e|$ denote the number of solution pairs $(x, c)$ modulo $p^e$ in the set $T_e$.
\end{definition}

\begin{proposition} \label{miniproofxc}
Let $p$ be an odd prime. Then $|T_1| = m^2$ for $x \in \{1, 2, \ldots, mp \}$, and for $c \in \{1, 2, \ldots, m\}$.
\end{proposition}

\begin{proof}
Consider $f(x, c) = g^{x-1+c} - x$. We first want to find the number of solutions modulo $p$, where $f(x, c) \equiv 0 \pmod{p}$. From Theorem~\ref{thm:vset} we know that the size of the value set, $V = \{f(p, c) \pmod{p} \mid 1\leq c \leq m\}$ gives us the number of $(x, c)$ solutions for $x \in \{1, 2, \ldots, p\}$. We know that the elements in $V$ are unique because they are simply powers of $g$, as shown in Lemma~\ref{lem:vset}. So, there will be $m$ unique values of $x \in V$.

Expand the range to $x \in \{1, 2, \ldots, p, p+1, \ldots, m\cdot p - 1, m\cdot p\}$. Since we are concerned with solutions pairs of $x$ and $c$, we want to ensure that the solution pairs do not repeat after expanding our range of $x$. 

Consider values of $x$ such that $x_0 \equiv x_1 \ldots \equiv x_n$, and that are defined such as in Theorem~\ref{thm:x0c0soln}. Further, Theorem~\ref{thm:x0c0soln} implies that each $x_i$ will get a unique value of $c_i$, so that for all $(x_i, c_i)$ solution pairs, $c_i \neq c_j$. We know this because of periodicity of $c$. In particular, there will be $m$ different values of $c$ that correspond to $m$ different multiples of $x$. 
So when we expand the range of $x$ to $\{1, 2, \ldots, p, p+1, \ldots, m\cdot p - 1, m\cdot p\}$, number of unique solutions will increase by a multiple of $m$. Thus $|T_1| = m^2$.

\end{proof}

\begin{lemma}\label{lemmaonesoln}
Let $p$ be an odd prime, and consider a fixed $x_0 \in \mathbb{Z} / m\mathbb{Z}$,
and $g \in \mathbb{Z}_p$, $p \nmid g$. Consider the function $f(x, c) = 
\omega(g)^{x_0} \oneunit{g}^{x-1+c} - x $, and let
\begin{equation*}
|N_1| = | \{ (\bar{x}, \bar{c}) \in (\mathbb{Z}_p/p \mathbb{Z}_p)^\times \times (\mathbb{Z}_p/p \mathbb{Z}_p)) \mid f(x, c) \equiv 0 \pmod{p} \}|
\end{equation*}

then $|N_1| = p$.
\end{lemma}
\begin{proof}

Consider the power series representation of the function, $f(x, c)$ as follows:
\begin{align*}
f(x, c) &= \omega(g)^{x_0} \oneunit{g}^{x-1+c} - x \\
&= \omega(g)^{x_0}\exp((x-1+c)\log \oneunit{g}) - x\\
&= \omega(g)^{x_0}(1 + (x-1+c) \log( \oneunit{g}) + (x-1+c)^2 \log( \oneunit{g})/ 2! \\
&\text{ } \text{ } + \text{ higher-order terms in powers of} \log(\oneunit{g}) ) - x \\
&\equiv \omega(g)^{x_0} - x \pmod{p}
\end{align*}

Similarly to Theorem~\ref{x0fixedpt}, we observe that if $f(x, c) \equiv 0 \pmod{p}$, then $\omega(g)^{x_0} \equiv x \pmod{p}$. Since $x_0$ and $g$ are fixed, and $c$ is free to be anything, there will be $p$ solutions to this equation. Thus, $|N_1| = p$.

\end{proof}

\begin{proposition}\label{prop:hlnoc}
Let $p$ be an odd prime, $g \in \mathbb{Z}_p$, $p \nmid g$, and consider a fixed $x_0 \in \mathbb{Z}/m\mathbb{Z}$. 
Consider the function $f(x, c) = \omega(g)^{x_0} \oneunit{g}^{x-1+c} - x $, and let
\begin{equation*}
|N_e| = | \{ (\bar{x}, \bar{c}) \in (\mathbb{Z}_p/p^e \mathbb{Z}_p)^\times \times (\mathbb{Z}_p/p^e \mathbb{Z}_p)) \mid f(x, c) \equiv 0 \pmod{p^e} \}|
\end{equation*}
then \begin{equation*}
|N_e| = p^{e-1} |N_1| .
\end{equation*}

\end{proposition}

\begin{proof}
First recall that we can expand the function, $f(x, c)$ as:
\begin{align*}
f(x, c) &= \omega(g)^{x_0} \oneunit{g}^{x-1+c} - x \\
&= \omega(g)^{x_0}(1 + (x-1+c) \log( \oneunit{g}) + (x-1+c)^2 \log( \oneunit{g})/ 2! \\
&\text{ } \text{ } + \text{ higher-order terms in powers of} \log(\oneunit{g}) ) - x \\
&\equiv \omega(g)^{x_0} - x \pmod{p}
\end{align*}

Now, we consider the partial derivatives of $f(x, c)$ where we note that $\log\oneunit{g} \equiv 0 \pmod{p}$, since $\oneunit{g} \equiv 1 \pmod{p}$, thus obtaining:
\begin{align*}
\frac{\partial f }{ \partial x} (x, c)& = \omega(g)^{x_0} (\exp((x-1+c)\log \oneunit{g}) \cdot \log \oneunit{g}) - 1 \\
&\equiv -1 \pmod{p}\\
\frac{\partial f}{\partial c} (x, c)&= \omega(g)^{x_0} (\exp((x-1+c)\log \oneunit{g}) \cdot \log \oneunit{g})\\
&\equiv 0 \pmod{p} .
\end{align*}

Observe that $\frac{\partial f}{\partial x}(x, c) \not\equiv 0 \pmod{p}$. Thus, we can apply a multivariable version of Hensel's lemma (see ~\cite[Proposition 3.4]{hold_rob}) to observe that there are $p^{e-1}$ possible ways to lift each solution modulo $p$ to a solution modulo $p^e$.

Thus, $|N_e| = p^{e-1}|N_1|$.
\end{proof}

\begin{theorem}\label{thm:finalxcsoln}
Let $p$ be an odd prime, and fix $g$ where $p \nmid g$. Then $|T_e| = p^{e-1}|T_1| = m^2p^{e-1}$ is the number of solutions to the congruence
\begin{equation*}
g^{x-1+c} \equiv x \pmod{p^e}
\end{equation*}
for $(x, c)$ such that $x \in \{1, 2, \ldots, mp^e\}$, and $c \in \{1, 2, \ldots, mp^{e-1} \}$. Further the set of $x$ that solve this equation are all distinct modulo $p^e$.
\end{theorem}

\begin{proof}
We use Chinese Remainder Theorem to count the number of possible $x_0$ and $c_0$ modulo $m$. Here $x_0$ is defined by interpolation in Section~\ref{interpolation} and $c_0$ is a solution pair to $x_0$. Note that $c_0$ can be chosen to be any value modulo $m$ (see Theorem~\ref{thm:x0c0soln} and Corollary~\ref{corol:rptc}).

By interpolation, we know that $g^{x-1+c} = \omega(g)^{x_0}\oneunit{g}^x \equiv x \pmod{p^e}$, since $x - 1 + c \equiv x_0 \pmod{m}$. 
So rewrite the equation and consider the following equivalences:
\begin{align*}
x-1+c &\equiv x_0 \pmod{m} \\
x &\equiv x_0 +1 - c \pmod{m}
\end{align*}
Also consider $x_1$, a solution in modulo $p^e$. So we have the following two equations to work with:
\begin{align}
x &\equiv x_0 +1 - c \pmod{m} \label{firstxeqn} \\
x &\equiv x_1 \pmod{p^e} \label{x_1_pe}
\end{align}

Apply the Chinese Remainder Theorem to the equations. There will be exactly one $x \in \mathbb{Z}/p^em\mathbb{Z}$ that solves equations~\ref{firstxeqn} and~\ref{x_1_pe} for each $x_0$ that we choose. Note that there are $m$ possible choices of $x_0$.

Similarly, we consider the set of equations for $c$ so we can apply Chinese Remainder Theorem. Consider $c \equiv c_0 \pmod{m}$ and $c = c_1 \pmod{p^e}$, where $c_1$ is the corresponding solution pair to $x_1$. By the Chinese Remainder Theorem, there will be exactly one $c \in \mathbb{Z}/p^em\mathbb{Z}$ that solves the previous two equations for each $c_0$ that we choose. Observe that there are $m$ possible choices of $c_0$.

Proposition~\ref{prop:hlnoc} and Lemma~\ref{lemmaonesoln} use Hensel's lemma to show that there are $p \cdot p^{e-1}$ choices for $(x, c)$ modulo $p^e$. So the combinations of $x_0$, $c_0$ and $(x, c)$ pairs modulo $p^e$ will yield $m^2 p^e$ pairs of solutions for $x \in \{1, 2, \ldots, mp^e \}$, and $c \in \{1, 2, \ldots, mp^e\}$.

Notice that the domain of $c$ is not the size that we want. By 
both Hensel's Lemma and Chinese Remainder Theorem we 
produced values of $c \in \{1, \ldots,mp^{e-1}, \ldots, mp^e\}$ 
but we want $c \in \{1, \ldots, mp^{e-1}\}$. However by 
Theorem~\ref{thm:repetition} we know that the period of $c$ is length 
$mp^{e-1}$. Thus, we can divide $mp^e$ by $p$. Hence, we obtain 
$|T_e| = m^2p^{e-1}$.

\end{proof}


\subsection{Considering $p = 2$}

Before we move ahead, we will  look at solutions modulo 2.

\begin{lemma} \label{p2oddsoln}
For fixed $c \in \Z$ and $g \in 2\Z+1$, all solutions $x \in \Z$ to the equation $f(\cdot, c)=g^{x-1+c} \equiv x \pmod{2^e}$ are odd.
\end{lemma}

\begin{proof}
Let $g \in 2\Z+1$ and $c \in \Z$ be fixed.
\begin{eqnarray*}
g^{x-1+c}-x &\equiv& 0 \pmod{2^e}
\\g^{x-1+c}-x &\equiv& 0 \pmod{2}.
\end{eqnarray*}
Since $g$ is odd, $g^{x-1+c}$ is also odd for $x-1+c \in \Z$. Then $g^{x-1+c} \equiv 1 \pmod{2}$ and we get
\begin{eqnarray*}
1-x &\equiv& 0 \pmod{2}
\\x &\equiv& 1 \pmod{2}.
\end{eqnarray*}
So all integer solutions $x$ are odd.
\end{proof}

\begin{theorem} \label{p2fixedpt} For $p = 2$, let $c \in \Z$ and $g \in \Z_2^\times$ be fixed.
  Then there is exactly one
  solution to each of the equations
$$ \oneunit{g}^{x-1+c} = x $$
and 
$$-\oneunit{g}^{x-1+c} = x $$
for $x \in 1+2\Z_2$.
\end{theorem}

\begin{proof} As with the analogous proof for odd primes, we start by finding solutions modulo $p=2$.  Since
  $\oneunit{g} \equiv 1 \pmod{2}$, both equations reduce to 
$$1 \equiv x \pmod{2}.$$
This expression clearly has exactly one solution, and note that this expression agrees with Lemma~\ref{p2oddsoln}.

Since we know that
  $\oneunit{g}$ is in $1+4\Z_2$, we have
  that
\begin{eqnarray*}
\oneunit{g}^{x-1+c}=\oneunit{g}^{c-1}\oneunit{g}^x&=& 
\oneunit{g}^{c-1}(\exp(x \log(\oneunit{g})) \\
=\oneunit{g}^{c-1}\ \ (1&+&x\log(\oneunit{g})+
x^2\log( \oneunit{g})^2/2! \\
&+& \mbox{higher order terms in powers of }
\log(\oneunit{g}))
\end{eqnarray*}
where from ~\cite[Proposition 4.5.9]{gouvea}, we know that 
$\log(\oneunit{g}) \in 4\Z_2$.  
Now we have a convergent
power series since $|\log(\oneunit{g})^i/i!|_2 \to 0$ as $i \to \infty$ (see ~\cite[Chapter 2, Thm 3.1]{bachman}), and we will look at $f(\cdot,c)$ and its 
derivative to see if we can apply Hensel's lemma. 

To count solutions to $F_0(x) \equiv \oneunit{g}^{x-1+c} \equiv x$,
we let $a = 1$, and let
\begin{eqnarray*}
f(\cdot,c)= F_0(x)-x=\oneunit{g}^{c-1}(1&+&x\log(\oneunit{g})+
x^2\log( \oneunit{g})^2/2! \\
&+& \mbox{higher order terms in powers of }
\log(\oneunit{g})) - x
\end{eqnarray*}

Then we have 
\begin{eqnarray*}
f(a,c)\equiv(1)(1&+&1(0)+
1^2(0) \\
&+& \mbox{higher order terms equivalent to 0 (mod 2) }
) -  1 \pmod{2}
\end{eqnarray*}
\begin{eqnarray*}
\equiv1-1
\equiv 0 \pmod{2}
\end{eqnarray*}
 Also, since we know $\log{(\oneunit{g})}\in 4\Z_2$, $\log{(\oneunit{g})} \equiv 0 \pmod{2}$, and we have
\begin{eqnarray*}
f'(\cdot,c)=\oneunit{g}^{c-1}(\log{(\oneunit{g})}+x\log{(\oneunit{g})}^2+x^2\log{(\oneunit{g})^3/2!}+...)-1
\end{eqnarray*}
\begin{eqnarray*}
f'(a,c)=\oneunit{g}^{c-1}(\log{(\oneunit{g})}+\log{(\oneunit{g})}^2+...)-1 \equiv -1 \not\equiv 0 \pmod{2}
\end{eqnarray*}
which is also convergent (see \cite[Proposition 4.4.4]{gouvea}).
Now we know we can apply a generalization of Hensel's lemma (see ~\cite[Cor. 3.3]{hold_rob}), 
which states that there is a unique $x \in \Z_2$ for which $x \equiv 1 \pmod{p}$ and $f(x)=0$ in $\Z_2$.

Note that similar steps can be used to show there is one solution in $\Z_2$ to $F_1(x) \equiv -\oneunit{g}^{x-c+1} \equiv x$ as well.

\end{proof}

\begin{corollary} \label{crtfixedpt2} For $p=2$, let $g,c \in \Z$ be fixed.
  Then there is exactly $1$ solution to the congruence
\begin{equation*}
g^{x-1+c} \equiv x \pmod{2^e} 
\end{equation*}
for $x \in \{1,2, \ldots, 2^e \}$. 
\end{corollary}

\begin{proof}

Having determined that $x$ is odd by Lemma~\ref{p2oddsoln}, we know that $x-1+c \equiv 1-1+c \equiv c \pmod{2}$. 
Because c is fixed and our functions $F_0(x)$ and $F_1(x)$ are defined on $x \in 1+2\Z_2$,
we only need to count solutions to $F_1(x)$ if $\oneunit{g} \in 3+4\Z_2$ and $c$ is odd and $F_0(x)$ otherwise.
The number of solutions where $x$ is odd in the correct equation will be the same as the number of solutions to $f(\cdot,c)$.

Theorem~\ref{p2fixedpt} implies that there is exactly one
$x \in 1+2\Z_2$ 
in the appropriate function $F_0(x)$ or $F_1(x)$ that we have chosen based on $c$.
Note that since each lifting in the proof is equivalent to $1$ modulo $2$, we know $f(\cdot,c)$ and $f'(\cdot,c)$ are still defined, and $x-1+c$ is still equivalent to $c$,
 we can use the same function after each lift. So we get the unique solution to our congruence, $g^{x-1+c} \equiv x \pmod{2^e}$.

\end{proof}


\section{Conclusion}

Most of the previous analysis on the Welch equation, $g^{x-1+c} \equiv x \pmod{p^e}$, has looked at solutions only modulo $p$ 
and for $x \in \{1, 2, \ldots, p\}$, and the conclusions about the number of solutions on this range are mainly statistical (see ~\cite{drakakis}, for example).
We have found here that there are clear patterns for this equation modulo $p^e$ when we extend the range of $x$ to $x \in \{1, 2, \ldots, p^em\}$, 
where $m$ is the order of $g$ modulo $p$. Specifically, there are always $m$ solutions (for all primes) on this range when we fix $c$, and $m^2p^{e-1}$ solutions (for odd primes)
when we consider $c \in \{0, 1, \ldots, p^{e-1}m\}$ as an additional variable. If we can find how these solutions
are distributed on subintervals of length $p^e$, we may have a better understanding of what happens on the original range of $x$ from $1$ to $p^e$.

The value set described in Theorem~\ref{thm:vset} (i.e. $f(p,c) \equiv g^{p-1+c}-p \pmod{p}$) also helps us find which values of $x$ in the range from
$1$ to $p$ are solutions modulo $p$. This is especially helpful in the cases where $g$ is not a primitive root, since there is a
smaller value set that restricts the values of $x$ that may be a solution.

Finally, there are several other patterns found in analyzing the function $f(\cdot,c) \equiv g^{x-1+c}-x \pmod{p^e}$ that are left unexplored in this paper, such as the appearance of 
$p$ pairs of ``doubles," where $f(x,c)\equiv f(x+1,c) \pmod{p^e}$  when $g$ is a primitive root, whose investigation may help with understanding the distribution of solutions.

\section{Acknowledgements}
Both authors would like to thank Joshua Holden and Margaret M. Robinson for their endless support and guidance during their summer REU.


\end{document}